\documentclass[microtype]{gtpart}     
\usepackage{algorithmicx}
\usepackage{algorithm}
\usepackage{algpseudocode}
\algnewcommand\algorithmicinput{\textbf{INPUT:}}
\algnewcommand\INPUT{\item[\algorithmicinput]}
\algnewcommand\algorithmicoutput{\textbf{OUTPUT:}}
\algnewcommand\OUTPUT{\item[\algorithmicoutput]}
\usepackage{sidecap}
\RequirePackage[OT1]{fontenc}
\RequirePackage{verbatim, amscd,epsfig,amsmath,amsthm,amstext,amssymb,tikz}
\usepackage{mathrsfs, xypic}

\newcommand{\Rips}{\mathcal{R}}
\renewcommand{\P}{\mathcal{P}}
\renewcommand{\O}{\mathcal{O}}
\renewcommand{\D}{\mathcal{D}}
\newcommand{\F}{\mathbb{F}}

\newcommand{\M}{\mathcal{M}}

\newcommand{\rank}{\operatorname{rank}}
\renewcommand{\R}{\mathbb{ R}}

\newcommand{\im}{\operatorname{im}}

\newcommand{\sym}{\operatorname{sym}}
\DeclareMathOperator{\id}{id}

\usepackage{url}
\usepackage{enumerate}

\title{Rips filtrations for quasi-metric spaces and asymmetric functions with stability results}

\author{Katharine Turner}
\givenname{Katharine}
\surname{Turner}
\address{Mathematical Sciences Institute\\
Australian National University\\
Canberra, ACT\\
Australia}
\email{katharine.turner@anu.edu.au}
\urladdr{}


\arxivreference{1608.00365}

\numberwithin{equation}{section}
\newtheorem*{proposition*}{Proposition}
\newtheorem{theorem}{Theorem}
\newtheorem{cor}[theorem]{Corollary}

\theoremstyle{definition}
\newtheorem{definition}[theorem]{Definition}

\newtheorem{proposition}[theorem]{Proposition}
\newtheorem{lemma}[theorem]{Lemma}
\newtheorem*{thm*}{Theorem}
\newtheorem*{theorem*}{Theorem}

\theoremstyle{definition}
\begin{document}

\begin{abstract}    
The Rips filtration over a finite metric space and its corresponding persistent homology are prominent methods in topological data analysis to summarise the ``shape'' of data. Crucial to their use is the stability result that says if $X$ and $Y$ are finite metric spaces then the (bottleneck) distance between the persistence diagrams constructed via the Rips filtration is bounded by $2d_{GH}(X,Y)$ (where $d_{GH}$ is the Gromov-Hausdorff distance). A generalisation of the Rips filtration to any \emph{symmetric} function $f:X\times X\to \R$ was defined by Chazal, De Silva and Oudot in \cite{CSO}, where they showed it was stable with respect to the correspondence distortion distance. Allowing asymmetry, we consider four different persistence modules, definable for pairs $(X,f)$ where $f:X\times X\to \R$ is any real valued function. These generalise the persistent homology of the symmetric Rips filtration in different ways. The first method is through symmetrisation. For each $a\in [0,1]$ we can construct a symmetric function $\sym_a(f)(x,y)=a\min \{d(x,y), d(y,x) \}+ (1-a)\max \{d(x,y), d(y,x)\}$. We can then follow the apply the standard theory for symmetric functions and get stability as a corollary.
The second method is to construct a filtration $\{\Rips^{dir}(X)_t\}$ of ordered tuple complexes where tuple $(x_0, x_2, \ldots x_p)\in \Rips^{dir}(X)_t$ if $d(x_i, x_j)\leq t$ for all $i\leq j$. Both our first two methods have the same persistent homology as the standard Rips filtration when applied to a metric space, or more generally to a symmetric function. We then consider two constructions using an associated filtration of directed graphs or preorders. For each $t$ we can define a directed graph $\{D(X)_t\}$ where directed edges $x\to y$ are included in $D(X)_t$ whenever $\max\{f(x,y),f(x,x),f(y,y)\}\leq t$ (note this is when $d(x,y)\leq t$ for $f=d$ a quasi-metric). From this we construct a preorder where $x\leq y$ if there is a path from $x$ to $y$ in $D(X)_t$. We build persistence modules using the strongly connected components of the graphs $D(X)_t$, which are also the equivalence classes of the associated preorders. We also consider persistence modules using a generalisation of poset topology to preorders.

The Gromov-Hausdorff distance, when expressed via correspondence distortions, can be naturally extended as a correspondence distortion distance to set-function pairs $(X,f)$. We prove that all these new constructions enjoy the same stability as persistence modules built via the original persistent homology for symmetric functions. 
\end{abstract}

\maketitle


\section{Introduction}

The Rips filtration over a finite metric space $(X,d)$ is a filtration of simplicial complexes $\{\Rips(X,d)_t\}_{t \in [0,\infty)}$, where $\Rips(X,d)_t$ is the clique complex over the graph whose vertex set is $X$ and edge set  $\{[x,y]:d(x,y)\leq t\}$. It adds topological structure to an otherwise disconnected set of points. The persistent homology of the Rips filtration is widely used in topological data analysis because it encodes useful information about the geometry and topology of the underlying metric space (\cite{chazal2009gromov, ghrist2008barcodes, lee2011discriminative, xia2014persistent}). There are many potential applications for studying data whose structure is a quasi-metric space. Examples includes the web hyperlink quasi-metric space, road networks, and quasi-metrics induced from weighted directed graphs found throughout science (for example, biological interaction graphs \cite{klamt2009computing}  or the connections in neural systems \cite{kaiser2011tutorial, BBP}).
More generally we wish to define and show stability of Rips filtrations for sublevel sets of any (not necessarily symmetric) function $f:X\times X\to \R$.

Historically the Rips filtration was defined as a special increasing family of simplicial complexes built from a finite metric space. A metric space is a set $X$ equipped with a distance function $d:X\times X \to \R$ that satisfies the following properties:
\begin{itemize}
\item[(i)] $d(x,y)\geq 0$ for all $x\in X$ ($d$ is \emph{non-negative})
\item[(ii)] $d(x,y)=d(y,x)$ for all $x,y\in X$ (that is the function $d$ is \emph{symmetric})
\item[(iii)] $d(x,z)\leq d(x,y) + d(y,z)$ (that $d$ satisfies the \emph{triangle inequality})
\item[(iv)] $d(x,y)=0=d(y,x)$ if and only if $x=y$ 
\end{itemize}
For any $r\geq 0$ we define the Rips complex of $X$ at length scale $r$, denoted as $\mathcal{R}(X,d)_r$, as the abstract simplicial complex where $[x_0, x_1, \ldots x_k]\in \mathcal{R}(X,d)_r$ whenever $d(x_i, x_j)\leq r$ for all $i,j$. We can think of $\mathcal{R}(X,d)_r$ as adding a topological structure of length scale $r$. It is easy to check that if $r\leq s$ then $\mathcal{R}(X,d)_r\subset \mathcal{R}(X,d)_s$. We thus can define the Rips filtration of $X$ as the increasing family of simplicial complexes $\{\mathcal{R}(X,d)_r\}_{r\in [0,\infty)}$.

Two classic types of examples of Rips filtrations are examples that come from finite point clouds sitting inside some larger space (such as Euclidean space), and examples built from graphs. If $X\subset \R^{d}$ is a set of points then it inherits a finite metric space structure from that of $\R^d$, the distance function is just the restriction of the Euclidean distance function to the set $X$. Given a graph $G$ (with or without lengths on the edges) we can let the vertices of the graph be the finite set $X$ and then construct a distance function on $X$ by defining $d(x,y)$ as the shortest path length of all the paths from $x$ to $y$ in $G$. 

From the Rips filtration we can produce a persistence module which describes its persistent homology. 
A persistence module is a family of vector spaces $\{V_t:t\in \R\}$ equipped with linear maps $\phi_s^t:V_s\to V_t$ for each pair $s\leq t$, such that $\phi_t^t=\id$ and $\phi_s^t=\phi_s^r\circ \phi_r^t$ whenever $s\leq r\leq t$. The persistence module we construct from the persistent homology of a Rips filtration over $(X,d)$ has vector spaces $\{H_*(\Rips(X,d)_t)\}_{t \in [0,\infty)}$ along with maps on homology induced by inclusions; $\phi_s^t=\iota_*:H_*(\Rips(X,d)_s) \to H_*(\Rips(X,d)_t)$ when $s\leq t$.

Arguable the most important theoretical results in topological data analysis are the stability theorems. These stability results come in variety of forms but generally say that if two sets of input data are close then various persistence modules computed from them are also close. To be specific we need to quantify what is meant by ``close'' for these different kinds of objects.

We can measure how close persistent modules are via whether there exist suitable families of interleaving maps. This distance is closely related to the bottleneck distance between the corresponding persistence diagrams or barcodes. Two persistence modules, $(\{V_t\}, \{\phi_s^t\})$ and $(\{U_t\}, \{\psi_s^t\})$, are called $\epsilon$-interleaved when there exist families of linear maps $\{\alpha_t:V_t \to U_{t+\epsilon}\}$ and $\{\beta_t:U_t\to V_{t+\epsilon}\}$ satisfying natural commuting conditions. There is a pseudo-metric on the space of persistence modules called the interleaving distance, $d_{int}$, which is the infimum of the set of $\epsilon>0$ such that there exists an $\epsilon$-interleaving. 
More details about the interleaving distance is provided in section \ref{sec:background}. In this paper we will be considering a variety of different persistence modules, but we will always use the interleaving distance to quantify ``closeness''.

Gromov-Hausdorff distance is a classical distance between metric spaces. There are many equivalent formulations of Gromov-Hausdorff distance but for the purposes of this paper we will focus on that using correspondences. The set $\M \subset X \times Y$ is a \emph{correspondence} between $X$ and $Y$ if for all $x\in X$ there exists some $y\in Y$ with $(x,y)\in \M$ and for all $y\in Y$ there is some $x\in X$ with $(x,y)\in \M$. Using correspondences we can define the Gromov-Hausdorff distance between $X$ and $Y$ as
\begin{align}
d_{GH}(X,Y) = \frac{1}{2} \inf_{\{\text{correspondences }\M\}}  \sup_{(x_1,y_1), (x_2, y_2) \in \M}|d_X(x_1, x_2) -d_Y(y_1, y_2)|.
\end{align}
Here $ \sup_{(x_1,y_1), (x_2, y_2) \in \M}|d_X(x_1, x_2) -d_Y(y_1, y_2)|$ is the distortion of the correspondence $\M$. We can define the \emph{correspondence distortion distance} between set-function pairs $(X,f:X\times X\to \R)$ and $(Y,g:Y \times Y\to \R)$
by $$d_{CD}((X,f),(Y,g)) = \frac12 \inf_{\M \text{ correspondence between $X$ and $Y$}}  \sup_{(x_1,y_1), (x_2, y_2) \in \M}|f(x_1, x_2) -g(y_1, y_2)|.$$
This agrees with the standard definition for the Gromov-Hausdorff distance when $(X,d_X)$ and $(Y,d_Y)$ are metric spaces. More background and details about the correspondence distortion distance are presented in section \ref{sec:metric}.

%
%
%
%
%


Useful as the Rips filtration for finite metric spaces is, there are scenarios where the input is not a finite metric space. For example, it is common in data analysis to considers data sets $X$ equipped with a dissimilarity measure. A dissimilarity measure is a map $d_X : X \times X\to \R$ that satisfies $d_X(x,x)=0$ and $d_X(x,‰y) = d_X(y,x)$ for all $x,y \in X$, but is not required to satisfy any of the other metric space axioms. In \cite{CSO}, Chazal, De Silva and Oudot generalised the notion of a Rips filtration to cover dissimilarity measures and more generally for any symmetric function $f:X\times X\to \R$. Just as in the finite metric space case, the Rips complex of $X$ with parameter $r$, denoted as $\mathcal{R}(X,f)_r$, is defined as the abstract simplicial complex where $[x_0, x_1, \ldots x_k]\in \mathcal{R}(X,f)_r$ whenever $f(x_i, x_j)\leq r$ for all $i,j$ (including $i=j$).

Persistent homology can be applied to any increasing family of topological spaces so it is then natural to define persistence modules from the persistent homology of Rips filtrations built from any symmetric function. This was shown to be stable in \cite{CSO}.

\begin{theorem*}
Let $f:X\times X \to \R$ and $g:Y\times Y\to \R$ be symmetric functions and $\Rips(X,f), \Rips(Y,g)$ their corresponding Rips filtrations. If $d_{CD}((X,f), (Y,g))$ is finite then for all $\epsilon>d_{CD}((X,f), (Y,g))$, the $k$th homology persistence modules of $\Rips(X,f)$ and $\Rips(Y,g)$ are $\epsilon$-interleaved. In particular, when $(X, d_Y)$ and $(Y, d_Y)$ are compact metric spaces then $\Rips(X,d_X)$ and $\Rips(Y,d_Y)$ are $\epsilon$-interleaved for all $\epsilon>2d_{GH}(X,Y)$.
\end{theorem*}
 
The proofs of the interleaving results in \cite{CSO} didn't have any requirement on the function $f:X\times X\to \R$ except that it had to be symmetric. The purpose of this paper is to complete this generalisation procedure to lose that symmetry requirement. However,  there are multiple ways to  how to use asymmetry information, and so we have explored a variety of different constructions. 

One method is to study related symmetric functions. We can take our original function $f$ and construct a parametric family of related symmetric functions $\sym_a(f)$ where $a\in [0,1]$ and 
\begin{align*}
\sym_a(f)(x,y)= a \min \{f(x,y), f(y,x)\} + (1-a) \max\{f(x,y), f(y,x)\}.
\end{align*}
We can then construct the Rips filtration as in \cite{CSO} for the set-function pair $(X,\sym_a(f))$. 
Notably if $f$ is a symmetric function to begin with then $\sym_a(f)=f$ for all $a\in [0,1]$ and hence this symmetrisation process does give a generalisation of Rips filtrations to any set-function pair. We can show that the correspondence distortion distance between $(X,\sym_a(f))$ and $(Y, \sym_a(g))$ is bounded by that between $(X,f)$ and $(Y,g)$. We gain stability for these persistence modules constructed through this symmetrisation process as a corollary. 

A limitation with using a filtration of simplicial complexes is that a simplex is an inherently symmetric object. An alternative is to use ordered tuple complexes (shortened to OT-complexes). An OT-complex $K$ is a sets of ordered tuples $(v_0, v_1, \ldots v_p)$ such that if $(v_0, v_1, \ldots v_p)\in K$ then $(v_0, v_1, \ldots, \hat{v_i}, \ldots,  v_p)\in K$ for all $i$. Note that repetitions of the $v_j$ are allowed.
Chain complexes, boundary maps, homology and persistent homology can analogously be defined for OT-complexes.
We will define the \emph{directed Rips filtration} of OT-complexes for $f:X\times X\to \R$, as the filtration $\{\Rips^{dir}(X,f)_t\}$ of ordered tuple complexes where tuple $(x_0, x_2, \ldots x_p)\in \Rips^{dir}(X,f)_t$ if $f(x_i, x_j)\leq t$ for all $i\leq j$. We call the persistence module produced using the OT-homology of the directed Rips filtration the \emph{directed Rips persistence module}.

For each simplicial complex there is a canonical OT-complexes with isomorphic homology groups. Furthermore, since these homology isomorphism commute with the maps induced by inclusion, the persistence modules of these corresponding complexes are also isomorphic. This will implies that these  directed Rips filtration is truly generalisations of the Rips filtration built from a symmetric function. We will prove that the persistence modules constructed from these Rips filtrations are stable with respect to the correspondence distortion distance.


The third generalisation considers connected components. The standard dimension $0$ homology can be viewed as the vector space whose elements are linear combinations of connected components in the $1$-skeleton (i.e. the graph containing the $0$ and $1$ cells). When working with directed graphs there are two notions of connected components; weakly and strongly connected. Completely analogous to the traditional connected components story we can consider vector space whose elements are formal linear combinations the equivalence classes of \emph{strongly} connected components in the directed graph which is the $1$-skeleton of the directed Rips filtration.

Given a function $f:X\times X\to \R$, for each real number $t$ we can create a directed graph $D_t$ related to the sublevel set $f^{-1}(-\infty, t]$. $D_t$ should have vertex set $\{x\in X\mid f(x,x)\leq t\}$ and directed edge set $\{x\to y\mid \max\{f(x,x), f(x,y), f(y,y)\}\leq t\}.$ We can not include a directed edge $x\to y$ just when $f(x,y)\leq t$ because of the closure conditions a directed graph has to satisfy. For each $t\in \R$ we have a vector space $V_t$ 
of the formal linear combinations the equivalence classes of strongly connected components of $D_t$. Whenever $s\leq t$ we have an inclusion map from $D_s\subset D_t$ which induces a linear map from $V_s$ to $V_t$. This process directly constructs a persistence module which we call the \emph{strongly connected components persistence module}. We prove that these persistence modules are stable with respect to the correspondence distortion distance. We also provide some psuedocode on how to compute the barcode decomposition of the strongly connected components persistence module using a modification of the union find algorithm.

We also note that the persistence modules generated from formal linear combinations of the weakly connected components has already been covered as the dimension $0$ persistent homology of the filtration by sublevel sets of $\sym_1(f)$. 

Our fourth method uses the directed graphs described above to create a filtration of preorders. Given a directed graph $D$ over vertices $X$ we say $x\leq y$ if there is a path from $x$ to $y$. From a filtration of directed graphs we obtain a filtration of preorders. We then can construct persistence modules using poset topology (which can be generalised for all preorders, not just posets, discussed in the Appendix). We will call these \emph{preorder persistence modules}. We prove that these preorder persistence modules are stable with respect to the correspondence distortion distance. If $f:X\times X\to \R$ is a symmetric function, then the dimension $0$ preorder persistence module is the same as that of the persistent homology of the standard Rips filtration $\Rips(X,f)$ and its higher dimensional preorder persistence modules are always trivial. This implies that preorder persistence modules are describing asymmetry information.

\subsection{Related other works}

Other related work involves approaches in topological data analysis for incorporating asymmetry information. Ordered set homology is used in  \cite{BBP} in order to study the topology of brain networks. There have been a series of papers by Chowdhury and M\'emoli  (\cite{ChowdburyDowker, ChowdburyPath, ChowdburyConvergence}) about other constructions of persistence modules which incorporate asymmetry information.

\section{Directed graphs, Quasi and pseudo metric spaces and the correspondence distortion distance}\label{sec:metric}

The original stability result in topological data analysis for Rips filtrations was for filtrations of simplicial complexes built from metric spaces and the bound between persistence modules in terms of the Gromov-Hausdorff distance. This was generalised in \cite{CSO} to consider symmetric functions and the bound between the functions was the correspondence distance. However, there are many applications where asymmetry naturally arises of which important examples are quasi-metric spaces, such as those constructed as the path metric of some directed graph (with or without weights on the directed edges).

\begin{definition}
A \emph{directed graph} is a ordered pair $D=(V,A)$ where $V$ is a set whose elements are called \emph{vertices}, and $A$ is a set of ordered pairs of vertices called \emph{directed edges} or arrows. A \emph{weighted directed graph} is a directed graph where each arrow is given a non-negative weight.
\end{definition}

Note that a graph can be thought of as a directed graph such that whenever a directed edge $v\to w$ is in $A$ then we also its opposite direction $w\to v$ must also be in $A$. 

\begin{definition}\label{def:metrics}
Let $X$ be a non-empty set and $d:X \times X \to \R$. Consider the following potential properties of $d$:
\begin{enumerate}
\item $d(x,x')\geq 0$ for all $x,x' \in X$
\item $d(x,x')=d(x',x)$ for all $x,x' \in X$
\item For all $x,x' \in X$, $x=x'$ if and only if $d(x,x')=0$ and $d(x',x)=0$
\item $d(x, x'') \leq d(x, x') + d(x', x'')$ for all $x,x',x'' \in X$.
\end{enumerate}
If $(X,d)$ satisfies (1),(2), (3) and (4) it is called a metric space. If $(X,d)$ satisfies (1), (3) and (4) it is called a \emph{quasi-metric space} and we can call $d$ a \emph{quasi-metric}.
If $(X,d)$ satisfies (1), (2) and (4) it is called a \emph{pseudo-metric space} and we can call $d$ a \emph{pseudo-metric}.
If $(X,d)$ satisfies (1) and (4) it is called a \emph{pseudo-quasi-metric space} and we can call $d$ a \emph{pseudo-quasi-metric}.
\end{definition}

We can build examples of these different types of spaces using weighted directed graphs. Given a weighted directed graph $D=(V,A)$, and two vertices $x,y\in V$ we call $x=v_0, v_1, v_2, \ldots v_m=y$ a path from $x$ to $y$ if all of the arrows $v_i \to v_{i+1}$ are in $A$. The length of that path $(x=v_0, v_1, v_2, \ldots v_m=y)$ is the sum of the weights $\sum_{i=0}^{m-1} w(v_i\to v_{i+1})$. Construct $d:V\times V\to \R$  by setting $d(x,y)$ to be the length of the shortest path from $x$ to $y$ (and $\infty$ if no path exists). Since each arrow has non-negative in weight, the function $d$ automatically satisfies (1) in Definition \ref{def:metrics}. By considering the concatenation of paths, we can easily see that $d$ also automatically satisfies (4) in Definition \ref{def:metrics}. Thus $(V,d)$ must always be a quasi-pseudo metric space. 

More generally we can consider any function $f:X\times X \to \mathbb{R}$ not necessarily satisfying any of the properties $(1)$ to $(4)$. It is in this most general setting that we will prove stability theorems.

The \emph{Gromov-Hausdorff distance} between metric spaces $(X, d_X)$ and $(Y, d_Y)$ is often defined by
$$d_{GH}(X,Y) = \inf_{Z, f:X\to Z, g:Y\to Z} d_{H,Z}(f(X),g(Y))$$
where the infimum is taken over all metric spaces $Z$ and isometric embeddings $f$ and $g$ to $Z$ from $X$ and $Y$, respectively. $d_{H,Z}$ is the Hausdorff distance between subsets of $Z$. It is a standard result that the Gromov-Hausdorff distance is a metric on the space of compact metric spaces. 

A useful alternate, but equivalent, formula for the Gromov-Hausdorff distance can be given through correspondences. 
The set $\M \subset X \times Y$ is a \emph{correspondence} between $X$ and $Y$ if for all $x\in X$ there exists some $y\in Y$ with $(x,y)\in \M$ and for all $y\in Y$ there is some $x\in X$ with $(x,y)\in \M$. Using correspondences we know can write
\begin{align*}
d_{GH}(X,Y) = \frac{1}{2} \inf_{\M \text{ correspondence between $X$ and $Y$} }  \sup_{(x_1,y_1), (x_2, y_2) \in \M}|d_X(x_1, x_2) -d_Y(y_1, y_2)|.
\end{align*}

More generally, given functions $f:X\times X \to \R$ and $g:Y\times Y \to \R$ we can call
$dis_{(X,f), (Y,g)}(\mathcal{M})=\sup_{(x_1,y_1), (x_2, y_2) \in \M}|f(x_1, x_2) -g(y_1, y_2)|$ the \emph{distortion} of the correspondence $\M$. We can then define the correspondence distortion distance by minimising this correspondence distortion.

\begin{definition}\label{def:GH}
For set-function pairs $(X,f:X\times X\to \R)$ and $(Y,g:Y\times Y\to \R)$ the \emph{correspondence distance} between them can be defined as
\begin{align*}
d_{CD}((X,f),(Y,g)) &= \frac12 \inf_{\M \text{ correspondence between $X$ and $Y$}} dis_{(X,f), (Y,g)}(\mathcal{M})\\
&= \frac12 \inf_{\M \text{ correspondence between $X$ and $Y$}}  \sup_{(x_1,y_1), (x_2, y_2) \in \M}|f(x_1, x_2) -g(y_1, y_2)|.
\end{align*}
\end{definition}

This agrees with the standard definition for the Gromov-Hausdorff distance when $(X,d_X)$ and $(Y,d_Y)$ are metric spaces.
It is straightforward to verify that $d_{CD}$ is a pseudo-metric on the space of all set-function pairs and a metric on the space of finite quasi-metric spaces. The proofs are analogous to that for metric spaces discussed in \cite{Burago}.




\section{Background - persistence modules}\label{sec:background}

In this section we will cover some background theory on persistence modules and the interleaving distance between persistence modules. This is important because the interleaving distance between persistence modules bounds the bottleneck distance between their corresponding persistence diagrams.
To introduce and motivate the concepts we will provide a brief summary of the theory of persistent homology. We will omit most of the details as we will be phrasing all results in later sections in terms of persistence modules. For more details about the history and applications of persistent homology we refer the reader to \cite{weinberger2011, ghrist2008barcodes, edelsbrunner2008, carlsson2009}. 

Persistent homology describes how the homology groups evolve over an increasing family of topological spaces. Throughout this section let $K=\{K_t\}$ denote a family of reasonable topological spaces such that $K_s \subset K_t$ whenever $s\leq t$. Given $s\leq t$ the \emph{$k$-th dimensional persistent homology group} for $K$ from $s$ to $t$ are the $k^{th}$ dimensional homology classes in $K_s$ that ``persist'' until $K_t$, that is $Z_k(K_s)/(Z_k(K_t)\cup B_k(K_s))$. This is isomorphic to the image of the induced map on homology $\iota_*:H_k(K_s) \to H_k(K_t)$ from the inclusion $K_s \subset K_t$. 

Barcodes and persistence diagrams were introduced as discrete summaries of persistent homology information. Each barcode consists of a multiset of real intervals called bars. The barcode corresponding to the $k^{th}$ dimensional persistent homology of $K$ is  $\{I_1, I_2, \ldots I_n\}$ if, 
for all $s\leq t$, the dimension of $\im (\iota_*:H_k(K_s) \to H_k(K_t))$ equals the number of bars in $\{I_1, I_2, \ldots I_n\}$ that contain $[s,t)$. The corresponding persistence diagram is the multiset of points in $\R^2$, $\{(a_i,b_i)\}$, where $a_i$ and $b_i$ are the endpoints of the bar $I_i$, alongside infinite copies of every point along the diagonal (these diagonal points are acting the role of empty intervals).

Barcodes and persistence diagrams have played a prominent role in applied topology as topological summaries of data. In particular, they can provide insight into the ``shape'' of point cloud data through the persistent homology of the Rips filtration over that point cloud. Much of the power behind the use of barcodes and persistence diagrams comes from stability theorems, such as the stability theorem for the persistent homology of the Rips filtration over a finite metric space. 


Persistence, such as persistent homology of a filtration of simplicial complexes, can be defined directly at an algebraic level. In \cite{zomorodian2005computing}, Zomorodian and  Carlsson introduced the concept of a persistence module and proved that barcodes (and equivalently persistence diagrams) can be defined for persistence modules satisfying reasonable finiteness conditions. It was shown in \cite{chazal2009proximity} that we can define a distance between persistence modules (called the interleaving distance) and that the interleaving distance between persistence modules is a bound on the bottleneck distance of their corresponding persistence diagrams. Throughout this paper we will work directly with persistence modules. 

\begin{definition}
Let $R$ be a commutative ring with unity. A \emph{persistence module} over $A \subset \R$ is a family $\{P_t\}_{t\in A}$ of $R$-modules indexed by real numbers, together with a family of homomorphism $\{\iota_t^s:P_t\to P_s\}$ such that  $\iota_t^r=\iota_s^r \circ \iota_t^s$ for all $t \leq s\leq r$, and $\iota_t^t= \id P_t$. 
\end{definition}


If $R$ is a field then the $P_t$ are all vector fields and the $\iota_t^s$ are linear maps. As is standard in topological data analysis, we will assume throughout that $R$ is the fixed field $F$ (usually taken to be $\F_2$ for computational reasons). 
In the theory of persistence modules there are technical requirements about tameness. We say $\P$ is \emph{tame} if $\rank \iota_t^s$ is always finite for any $s<t$.  A sufficient condition for tameness is that $X$ is finite which is almost always true in any application. It is less straightforward in the constructions involving asymmetry to provide other nice sufficient conditions which would ensure the resulting persistent modules are tame (see the future directions). When the persistence modules are tame then the interleaving results will immediately imply a stability theorem for the persistence diagrams/barcodes. 

The space of persistent modules is a pseudo-metric space under the interleaving distance function. Here we will define the interleaving distance between two persistence modules as the infimum of $\epsilon>0$ such that they are $\epsilon$-interleaved. In this we slightly differ from \cite{chazal2009proximity} where they define both strongly and weakly $\epsilon$-interleaved, both of which are weaker than our notion of interleaving. More details about the pseudo-metric space structure of persistence modules and how the interleaving distance between persistence modules relates to the distances between corresponding persistence diagrams can be found in \cite{chazal2009proximity, ghrist2008barcodes, zomorodian2005computing}. 

\begin{definition}
Two persistence modules $\P^X$ and $\P^Y$ are \emph{$\epsilon$-interleaved} if there exist families of homomorphisms $\{\alpha_t:P^X_t \to P^Y_{t+\epsilon}\}_{t\in \R}$ and $\{\beta_t:P^Y_t \to P^X_{t+\epsilon}\}_{t\in \R}$ such that the diagrams in \eqref{eq:comdia1} and \eqref{eq:comdia2} commute.
 \end{definition}

\begin{equation}\label{eq:comdia1}
\xymatrix{P^X_t \ar[dr]^{\alpha_t} \ar[r]^\iota & P^X_{t'} \ar[dr]^{\alpha_t'} & \\ &  P^Y_{t+\epsilon}  \ar[r]_\iota   & P^Y_{t'+\epsilon}}
\qquad
\xymatrix{& P^X_{t+\epsilon} \ar[r]^\iota & P^X_{t'+\epsilon} \\  P^Y_{t}  \ar[r]_\iota \ar[ur]^{\beta_t}  & P^Y_{t'+\epsilon}\ar[ur]^{\beta_t'} }
\end{equation}

\begin{equation}\label{eq:comdia2}
\xymatrix{P^X_t \ar[dr]^{\alpha_t} \ar[rr]^\iota && P^X_{t+2\epsilon}\\  &  P^Y_{t+\epsilon}  \ar[ur]^{\beta_{t+\epsilon}}}
\qquad
\xymatrix{& P^X_{t+\epsilon} \ar[dr]_{\alpha_{t+\epsilon}}  \\  P^Y_{t}  \ar[rr]_\iota \ar[ur]^{\beta_t}  & & P^Y_{t+2\epsilon}}
\end{equation}

\begin{definition}
Two persistence modules $\P^X$ and $\P^Y$ are \emph{isomorphic} if they are $0$-interleaved.
\end{definition}

The diagrams in \eqref{eq:comdia1} and \eqref{eq:comdia2} are slightly different to those given in \cite{chazal2009proximity} but the diagrams here commuting will imply that theirs also commute. 


If $\P^X$ and $\P^Y$ are $\epsilon_1$-interleaved and $\P^Y$ and $\P^Z$ are $\epsilon_2$-interleaved then composing homomorphisms shows that $\P^X$ and $\P^Z$ are $(\epsilon_1+ \epsilon_2)$-interleaved. We can define a pseudo-distance on the space of persistence modules, called the \emph{interleaving distance}, where the interleaving distance between $\P^X$ and $\P^Y$ is the infimum of the set of $\epsilon>0$ such that $\P^X$ and $\P^Y$ are $\epsilon$-interleaved. It is worth noting that two persistence modules might have interleaving distance $0$ and yet not be $0$-interleaved (and thus not isomorphic).

%
%
%
%
%
\section{Existing stability results and Rips filtrations constructed from related symmetric functions}\label{sec:sym}

In this section we will recall the definition for the Rips filtration of a metric space and more generally for sublevel sets of a symmetric function $f:X\times X\to \R$. We will also recall the existing stability results for their persistent homology. Given a function $f:X\times X\to \R$ we construct a family of related symmetric functions $\sym_a(f)$  (for $a\in [0,1]$). We show that the persistent homology constructed from the $\sym_a(f)$ is stable as a corollary for the stability results for symmetric functions under the correspondence distortion distance.

\begin{definition}
Given a set $X$ and a symmetric function $f:X\times X\to \R$, the \emph{Rips filtration} of $(X,f)$ is a family of finite simplicial complexes $\Rips(X,f)=\{\Rips(X,f)_t\}_{t\geq 0}$ with $\Rips(X,f)_t$ the clique complex on the graph with vertices $X_t=\{x\in X: f(x,x)\leq t\}$ and edges $\{[x_1,x_2]\in X_t\times X_t: f(x_1, x_2)\leq t\}$. \footnote{Readers need to be warned that sometimes the Rips filtration is defined by adding the edge $[x_1, x_2]$ when $d_X(x_1, x_2)\leq t/2$ instead of $d_X(x_1, x_2)\leq t$, so sometimes results may differ from here by a corresponding  factor of $2$.}
\end{definition}


\begin{theorem}\label{the:Rips}
Let $f:X\times X \to \R$ and $g:Y\times Y\to \R$ be symmetric functions and $\Rips(X,f), \Rips(Y,g)$ their corresponding Rips filtrations. If $d_{CD}((X,f), (Y,g))$ is finite then for all $\epsilon>d_{CD}((X,f), (Y,g))$, the $k$th homology persistence modules of $\Rips(X,f)$ and $\Rips(Y,g)$ are $\epsilon$-interleaved. In particular, when $(X, d_Y)$ and $(Y, d_Y)$ are compact metric spaces then $\Rips(X,d_X)$ and $\Rips(Y,d_Y)$ are $\epsilon$-interleaved for all $\epsilon>2d_{GH}(X,Y)$.
\end{theorem}

Since the only condition required is symmetry of the filtration function, one approach for analysing general functions $f:X\times X \to \R$ is to construct related symmetric functions. We will consider a one-parameter family of possible symmetric filtrations.
We then prove stability for the Rips filtrations of these symmetric constructions in terms of the correspondence distortion distance between the original set-function pairs.

\begin{definition}\label{def:sym}
Let $(X,f)$ be a finite set $X=\{x_1, \ldots x_N\}$ equipped with function $f:X\times X\to \R$. For any $a\in [0,1]$ we can define a symmetric function 
\begin{align*}
	\sym_a(f):X \times X &\to \R\\
	(x,y)&\mapsto a\min\{f(x,y), f(y,x)\} + (1-a)\max\{f(x,y), f(y,x)\}
	\end{align*}
\end{definition}

Since $\sym_a(f)$ is symmetric we can construct its Rips filtration $\{\Rips(X, \sym_a(f))_t\}$ where $\Rips(X,\sym_a(f))_t$ is the simplicial complex containing $[x_0, x_1, \ldots x_p]$ whenever $sym_a(f)(x_i,x_j)\leq t$ for all $i,j$. We call this the \emph{Rips filtration under $\sym_a$}. If $f$ is a symmetric function then $\sym_a(f)=f$ for all $a$, which implies that the Rips filtration under $\sym_a$ generalises the symmetric Rips filtration.

As a corollary of the stability for symmetric functions we have stability for the symmetrised functions.

\begin{cor}\label{cor:symRips}
Fix $a\in [0,1]$ and a homology dimension $k$. Let $(X,f)$ and $(Y,g)$ be set-function pairs such that $d_{CD}((X,f),(Y,g))$ is finite. Let $P^X$ and $P^Y$ be the corresponding $k^{th}$ dimension homology persistence modules constructed from the corresponding Rips filtrations under $\sym_a $ ($\Rips(X,\sym_a(f))$ and $\Rips(Y,\sym_a(g))$ respectively). Then $d_{int}(P^X,P^Y)\leq 2d_{CD}((X,\sym_a(f)),(Y,\sym_a(g)))$.
\end{cor}

Unfortunately this method of constructing Rips filtrations is somewhat crude. We can show that in the process of symmetrising we dampen dissimilarities. This is not surprising as the space of symmetric functions is much smaller than that of functions generally.  In particular we will show in Theorem \ref{thm:symdampens} that $d_{CD}((X, \sym_a(f)),(Y, \sym_a(g)))\leq 2d_{CD}((X,f), (Y,g))$ for all $a\in [0,1]$. There are many examples where this inequality is strict. For asymmetric functions $d_{CD}((X, \sym_a(f)),(Y, \sym_a(g)))$ is often significantly smaller than $2d_{CD}((X,f), (Y,g))$. Suppose $X=Y$, $f:X\times X \to \R$  is an antisymmetric function and $g=-f$. Then by construction $\sym_a(f)=\sym_a(g)$ for all $a$ but for non-zero $f$, we generally have $d_{CD}((X,f),(X,-f))>0$.

The dampening process through symmetrisation is encapsulated in the following lemma.
\begin{lemma}\label{lem:sym}
Let $w,\hat{w},z, \hat{z}\in \R$. Then
\begin{enumerate}
\item[(i)] $|\max\{w,\hat{w}\}-\max\{z,\hat{z}\}|\leq \max\{|w-z|,|\hat{w}-\hat{z}|\}$
\item[(ii)] $|\min\{w,\hat{w}\}-\min\{z,\hat{z}\}|\leq \max\{|w-z|,|\hat{w}-\hat{z}|\}$
\end{enumerate}
\end{lemma}

\begin{proof}
We can prove (i) through a series of cases. 
If $w\leq \hat{w}$ and $z\leq \hat{z}$ then $|\max\{w,\hat{w}\}-\max\{z,\hat{z}\}|= |\hat{w}-\hat{z}|$.
If $w\geq \hat{w}$ and $z\geq \hat{z}$ then $|\max\{w,\hat{w}\}-\max\{z,\hat{z}\}|= |w-z|$.

If $w\leq \hat{w}$ and $z\geq \hat{z}$ then 
\begin{align*}
|\max\{w,\hat{w}\}-\max\{z,\hat{z}\}|&=|\hat{w}-z|\\
&\leq \begin{cases} |\hat{w}-\hat{z}| &\text{ if }\hat{z}\leq w\\
|w-z| &\text{ if } \hat{z}\geq w
\end{cases}\\
&\leq \max\{|w-z|,|\hat{w}-\hat{z}|\}
\end{align*}

Reversing the roles of the letters, we also see $|\max\{w,\hat{w}\}-\max\{z,\hat{z}\}|\leq \max\{|w-z|,|\hat{w}-\hat{z}|\}$ whenever $w\geq \hat{w}$ and $z\leq \hat{z}$

We can infer (ii) from (i) by replacing each of $w,\hat{w},z, \hat{z}$ by their negatives.
\end{proof}

\begin{theorem}\label{thm:symdampens}
Fix $a\in [0,1]$ and a homology dimension $k$. Let $(X,f)$ and $(Y,g)$ be set-function pairs. Then $d_{CD}((X, \sym_a(f)),(Y, \sym_a(g)))\leq 2d_{CD}((X,f), (Y,g))$.
\end{theorem}

\begin{proof}
It is sufficient to show that $dis_{(X,\sym_a(f)),(Y,\sym_a(g))}(\M) \leq dis_{(X,f),(Y,g)}(\M) $ for every correspondence $\M$. 

Fix some correspondence $\M \subset X \times Y$ and let $(x_1,y_1), (x_2, y_2) \in \M$. From Lemma \ref{lem:sym} (using $w=f(x_1, x_2), \hat{w}=f(x_2, x_1), z=g(y_1,y_2)$ and $\hat{z}=g(y_2,y_1)$) we know that both
\begin{align*}
|\min\{f(x_1,x_2), f(x_2,x_1)\}&-\min\{g(y_1,y_2), g(y_2,y_1)\}|\\
&\leq \max\{|f(x_1, x_2)-g(y_1, y_2)|, |f(x_2, x_1)-g(y_2, y_1)| \}
\end{align*}
and 
\begin{align*}
|\max\{f(x_1,x_2), f(x_2,x_1)\}&-\max\{g(y_1,y_2), g(y_2,y_1)\}|\\
&\leq \max\{|f(x_1, x_2)-g(y_1, y_2)|, |f(x_2, x_1)-g(y_2, y_1)| \}
\end{align*}
Taking a convex combination of these equations tells us that
\begin{align}\label{eq:pfsym}
|\sym_a(f)(x, \hat{x}) -\sym_a(g)(y, \hat{y})|\leq \max\{|f(x, \hat{x}) -g(y, \hat{y})|, |f(\hat{x}, x) -g(\hat{y}, y)|\}.
\end{align}
By taking the supremum on both sides over all pairs $\{(x,y), (\hat{x}, \hat{y})\} \in \M$ we see that 
$$dis_{(X,\sym_a(f)),(Y,\sym_a(g))}(\M) \leq dis_{(X,f),(Y,g)}(\M).$$
\end{proof}

\section{Persistent homology of OT-complexes}\label{sec:OT}

Ordered tuple complexes are an alternative to simplicial complexes. We will find them useful as they have more flexibility with regard to order; we can have asymmetric roles within the same tuple. 

 \begin{definition}
An \emph{ordered tuple} is a sequence of $(v_0, v_1, v_2, \ldots v_n)$ potentially including repeats. A \emph{ordered tuple complex} (shortened to \emph{OT-complex}) is a collection  $K$ of ordered tuples such that if $(v_0, v_1, v_2, \ldots v_n)\in K$ then $(v_0, \ldots, \hat{v_i}, \ldots v_n)\in K$ for all $i$ (where $(v_0, \ldots, \hat{v_i}, \ldots v_n)$ is the ordered tuple with $v_i$ removed).
\end{definition}

It is worth emphasising that each ordered tuple is determined by the ordered sequence and not just the underlying vertices; $(v_1, v_2, v_3)$ and $(v_3, v_1, v_2)$ are distinct and not even linearly related.

The ideas of homology and persistent homology naturally extend to OT-complexes. Thoughout $\F$ will be a fixed field.

\begin{definition}
Given an OT-complex $K$ we can build a chain complex $C_*(K)$ where $C_p(K)$ is the set of all the $\F$-linear combinations of the ordered tuples in $K$ with length $p+1$.  This is an $\F$-vector space whose basis vectors are the ordered tuples in $K$ of length $p+1$. We define a boundary map $\partial_{p}: C_{p}(K) \to C_{p-1}(K)$ by 
$$\partial_p((v_0, v_1, v_2, \ldots v_p))=\sum_{i=0}^k(-1)^{i}(v_0, \ldots, \hat{v_i}, \ldots v_p)$$
 and extending linearly. We define the $k$-th homology group of OT-complex $K$ as $H_k(K) = \ker (\partial_{k-1})/ \im(\partial_{k}).$ 
 \end{definition}
 
When $K_1\subset K_2$ are both OT-complexes then the inclusion of chains induces a map on their homology groups: $\iota_*:H_*(K_1)\to H_*(K_2)$.
 
\begin{definition} 
We say $\mathcal{K}=\{K_t\}$ is filtration of OT-complexes if $K_t \subset K_r$ whenever $t\leq r$. We define the \emph{$k$th dimensional ordered tuple persistence module} corresponding to $\mathcal{K}$  as follows:
\begin{itemize}
\item for each $t$ set the vector space $V_t=H_k(K)$ computed over $\F$ 
\item for each pair $s\leq t$ we have a linear map induced from inclusion $$\iota_{t\to s}:H_*(K_s) \to H_*(K_t).$$
\end{itemize}
\end{definition}
It is easy to check that this does satisfy the requirements of a persistence module.

We can define the directed Rips filtration as a filtration of OT-complexes where the condition for when a ordered tuple is included is dependent on the
order in which the points in the tuple appear. From this filtration of OT-complexes we can construct directed Rips persistence modules. 

\begin{definition}\label{def:OT-flag}
Let $(X,f)$ be a set-function pair. Set  $\{\Rips^{dir}(X,f)_t\}$ to be the  filtration of OT-complexes where  $(v_0, v_1, \ldots v_p)\in \Rips^{dir}(X,f)_t$ when $f(v_i, v_j)\leq t$ for all $i\leq j$. We call  $\{\Rips^{dir}(X,f)_t\}$ the \emph{directed Rips filtration} of $(X,f)$. 
For each dimension $k$, we will define the \emph{$k^{th}$ dimensional directed Rips persistence module} as the $k^{th}$ dimensional ordered tuple persistence module of $\{\Rips^{dir}(X,f)_t\}$.
\end{definition}

We claim that these directed Rips persistence modules are a generalisation of the Rips persistence modules constructed from symmetric functions. To do this we need to recall some classical relationships between the homology of OT-complexes and simplicial complexes. Indeed, a common first example of an OT-complex is via a simplicial complex. For a simplicial complex $K$ there is an OT-complex $K^{OT}$ where $(v_0, v_1, \ldots v_p)\in K^{OT}$ whenever $[v_0, v_1 \ldots v_p]$, after removing any repeats, is a simplex in $K$. In \cite{munkres1984elements}, Munkres calls the chain complex $C_*(K^{OT})$ the \emph{ordered chain complex of $K$}, and shows that the simplical homology of $K$ and the OT-complex homology of $K^{OT}$ are isomorphic. This isomorphism result holds also for persistence modules of filtrations of simplicial complexes as the isomorphisms on homology groups commute with the induced maps on homology by inclusions. This implies that if $f:X\times X\to \R$ is a symmetric function then the $k^{th}$ dimensional
 ordered tuple persistence module of $\{\Rips^{dir}(X,f)_t\}$ is isomorphic to the  $k^{th}$ dimensional persistence module of $\{\Rips(X,f)_t\}$.
 
 \subsection{Stability of the directed Rips persistence modules}

We will want to prove that the directed Rips persistence modules enjoy stability with respect to the correspondence distortion distance. To do this we will compare set-function pairs over different sets via their induced set-function pairs over a common set constructed via a fixed correspondence.

Given functions $f:X\times X \to \R$, and $g:Y\times Y \to \R$, along with a correspondence $\M \subset X\times Y$ we can pull back the functions $f$ and $g$ to corresponding functions on $\M\times \M$ via the projection maps;
\begin{align*}
f^{\M}:\M \times \M & \to \R\\
(x_1,y_1) \times (x_2, y_2)&\mapsto f(x_1, x_2).
\end{align*}
and
\begin{align*}
g^{\M}:\M \times \M & \to \R\\
(x_1,y_1) \times (x_2, y_2)&\mapsto g(y_1, y_2).
\end{align*}

The proof of the following lemma follows directly from the definitions of $f^{\M}$ and $g^{\M}$.

\begin{lemma}\label{lem:infdist}
Let $(X,f),(Y,g)$ be set-function pairs and $\M\subset X\times Y$ a correspondence. Then
$$\|f^{\M}-g^{\M}\|_\infty = 2dis_{(X,f), (Y,g)}(\mathcal{M})$$
\end{lemma}

We will also need to prove that the directed Rips persistence modules over $(X,f)$ and $(\M, f^\M)$ are isomorphic. To do this we will introduce the notion of the expansion of an OT-complex.

\begin{definition}
Let $K$ be an OT-complex. We say that $K$ is \emph{closed under adjacent repeats} if whenever $(v_0,v_1, \ldots v_p)\in C_p(K)$ then $(v_0, \ldots v_i, v_i, \ldots v_p)\in C_{p+1}(K)$ for all $i=0,1,\ldots p$.
 \end{definition}
 
It is worth observing that by construction $\{\Rips^{dir}(X,f)_t\}$ is closed under adjacent repeats for any set-function pair $(X,f)$.

\begin{definition}
Let $K$ and $\tilde{K}$ be OT-complexes, both closed under adjacent repeats, over vertex sets $V$ and $\tilde{V}$ respectively. We say that $\tilde{K}$ is an \emph{expansion} of $K$ if there exists a surjective map $\pi:\tilde{V} \to V$ and injective map $\iota:V\to \tilde{V}$ such that $\pi\circ \iota=\id_V$ and $(v_0, v_1, \ldots ,v_p)\in \tilde{K}$ if and only if  $(\pi(v_0), \pi(v_1), \ldots ,\pi(v_p))\in K$.

Let $\mathcal{K}=\{K_t\}$ and $\tilde{\mathcal{K}}=\{\tilde{K}_t\}$ be filtrations of OT-complexes over vertex sets $V$ and $\tilde{V}$ respectively. We say that $\tilde{\mathcal{K}}$ is an \emph{expansion} of $\mathcal{K}$ if there exists a surjective map $\pi:\tilde{V} \to V$  and injective map $\iota:V\to \tilde{V}$ such that $\pi\circ \iota=\id_V$ and, for all $t$,  $(v_0, v_1, \ldots, v_p)\in \tilde{K}_t$ if and only if  $(\pi(v_0), \pi(v_1), \ldots , \pi(v_p))\in K_t$.
\end{definition}

\begin{proposition}\label{prop:expansion}
If $\mathcal{K}=\{K_t\}$ and $\tilde{\mathcal{K}}=\{\tilde{K}_t\}$ are filtrations of OT-complexes such that $\tilde{\mathcal{K}}$ is an expansion of $\mathcal{K}$ then the OT persistence modules of $\mathcal{K}$ and $\tilde{\mathcal{K}}$ are isomorphic.
\end{proposition}
\begin{proof}
Without loss of generality we can relabel the points in $V$ to consider it as a subset of $\tilde{V}$ (relabelling $v\in V$ as $\iota(v)\in \tilde{V}$). In this case $\iota$ is the inclusion map and $\pi$ is a projection map.

Both $\pi:\tilde{K}_t \to K_t$ and $\iota:K_t\to \tilde{K}_t$ induce chain maps $\pi_\#: C_*(\tilde{K}_t) \to C_*(K_t)$ and $\iota_\#: C_*(K_t) \to C_*(\tilde{K}_t)$. Observe that  $\pi_\#\circ \iota_\#=\id:C_*(K_t)\to C_*(K_t)$, so $\pi_*\circ \iota_*=\id:H_*(K_t) \to H_*(K_t)$ for all $t$.

Suppose $(v_0, v_1, \ldots, v_i, \ldots, v_p) \in C_p(\tilde{K}_t)$. To construct a prism operator later we want to show that  $$(v_0, v_1, \ldots , v_i, \pi(v_i), \ldots , \pi(v_p))\in C_{p+1}(\tilde{K}_t).$$ To do this we use that $\tilde{K}_t$ is closed under adjacent repeats, the definition of expansions (twice), and the property that $\pi$ is a projection (so $\pi(\pi(v_j))=\pi(v_j))$.
\begin{align*}
(v_0, v_1, \ldots, v_i, \ldots, v_p) \in C_p(\tilde{K}_t) &\implies (v_0, v_1, \ldots, v_i, v_i, \ldots, v_p)\in C_{p+1}(\tilde{K})\\
&\implies  (\pi(v_0), \pi(v_1), \ldots, \pi(v_i), \pi(v_i), \ldots, \pi(v_p))\in C_{p+1}(K_t)\\
&\implies (\pi(v_0), \pi(v_1), \ldots, \pi(v_i), \pi(\pi(v_i)), \ldots, \pi(\pi(v_p)))\in C_{p+1}(K_t)\\
&\implies (v_0, v_1, \ldots , v_i, \pi(v_i), \ldots , \pi(v_p))\in C_{p+1}(\tilde{K}_t)
\end{align*}


Consider the prism operator
$$P((v_0, v_1, \ldots v_p))=\sum_{i=0}^p (-1)^{i} ((v_0, v_1, \ldots v_i, \pi(v_i), \pi(v_{i+1}), \ldots, \pi(v_p)).$$
Routine algebra shows that $\partial P + P\partial = i_\# \circ \pi_\#-\id$ and thus $i_\# \circ \pi_\#$ is chain homotopic to the identity. This implies $i_*\circ \pi_*:H_*(\tilde{K}_t) \to H_*(\tilde{K}_t)$ is the identity function.

The chain maps $\pi_\#$ and $i_\#$ commute with the inclusion maps for the filtrations of OT-complexes and hence 
%
%
the following diagrams commute

\begin{minipage}{0.4\linewidth}\centering
$$\xymatrix{
H_*(\tilde{K}_s)  \ar[d]^{\pi_*} \ar[r]^{\iota_*} & H_*(\tilde{K}_t) \ar[d]^{\pi_*}\\
H_*(K_s)  \ar[r]^{\iota_*}          &H_*(K_t) }$$
\end{minipage}
\begin{minipage}{0.4\linewidth}\centering
$$\xymatrix{
H_*(K_s)  \ar[d]^{i_*} \ar[r]^{\iota_*} & H_*(K_t) \ar[d]^{i_*}\\
H_*(\tilde{K}_s)  \ar[r]^{\iota_*}          &H_*(\tilde{K}_t) }$$
\end{minipage}

Since $i_*\circ \pi_*=\id:H_*(\tilde{K}_t) \to H_*(\tilde{K}_t)$ and  $\pi_*\circ i_*=\id:H_*(K_t) \to H_*(K_t)$ for all $t$ we see that
 $\mathcal{K}$ and $\tilde{\mathcal{K}}$ are isomorphic.
 \end{proof}

\begin{theorem}\label{thmDirRips}
Let $(X,f)$ and $(Y,g)$ be set function pairs such that $d_{CD}((X,f),(Y,g))$ is finite. Let $P^X$ and $P^Y$ be the corresponding $k$th dimension homology persistence modules constructed from the corresponding directed Rips filtrations $\{\Rips^{dir}(X,f)_t\}$ and $\{\Rips^{dir}(Y,g)_t\}$. $d_{int}(P^X, P^Y)\leq 2d_{CD}((X,f),(Y,g))$.
\end{theorem}
\begin{proof}
Since $d_{CD}((X,f),(Y,g)) $ is finite there exists some correspondence $\M\subset X\times Y$ with $dis_{(X,f), (Y,g)}(\M)$ finite. Fix a correspondence $\M\subset X\times Y$ with $dis_{(X,f), (Y,g)}(\M)$ finite. From this correspondence construct directed Rips filtrations $\{\Rips^{dir}(\M,f^\M)_t\}$ and $\{\Rips^{dir}(\M,g^\M)_t\}$ with corresponding  $k$-th dimensional persistence modules $P^{(X,\M)}$ and $P^{(Y,\M)}$.

By construction $\{\Rips^{dir}(\M,f^{\M})_t\}$ is an expansion of $\{\Rips^{dir}(X,f)_t\}$ and thus by Proposition \ref{prop:expansion} we know that the persistence modules $P^X$ and $P^{(X,\M)}$ are isomorphic. Similarly we can also show that  $P^Y$ and $P^{(Y,\M)}$ are isomorphic.

By Lemma \ref{lem:infdist} we know $\|f^{\M}-g^{\M}\|_\infty \leq 2dis_{(X,f), (Y,g)}(\M)$. There is an inclusion $$\Rips^{dir}(\M,f^\M)_t\subset \Rips^{dir}(\M,g^\M)_{t+2dis_{(X,f), (Y,g)}(\M) }$$ for all $t$ as
\begin{align*}
(v_0, v_1, \ldots v_n)&\in \Rips^{dir}(\M,f^\M)_t\\
&\implies f^\M(v_i, v_j)\leq t \text{ for all }i\leq j\\
&\implies g^{\M}(v_i, v_j)\leq t +dis_{(X,f), (Y,g)}(\M)  \text{ for all }i\leq j\\
&\implies (v_0, v_1, \ldots v_n)\in \Rips^{dir}(\M,g^\M)_{t+2dis_{(X,f), (Y,g)}(\M)}.
\end{align*}

Symmetrically there are also inclusions $\Rips^{dir}(\M,g^\M)_t\subset \Rips^{dir}(\M,f^\M)_{t+2dis_{(X,f), (Y,g)}(\M) }$ for all $t$. These inclusion maps induce a $2dis_{(X,f), (Y,g)}(\M)$ interleaving between $P^{(X,\M)}$ and $P^{(Y,\M)}$. This implies that $P^{X}$ and $P^{Y}$ are  $2dis_{(X,f), (Y,g)}(\M)$ interleaved.

By considering the infimum of the interleavings constructed by correspondences we see that $P^X$ and $P^Y$ is at most $2d_{CD}((X,f),(Y,g))$.
\end{proof}

\subsection{Comparison to ordered-set persistent homology}\label{sec:OS}

It is  possible to construct homology groups and persistence modules using ordered-sets instead of order-tuples. As a preemptive attempt to reduce confusion, this section will compare this ordered-tuple persistent homology to ordered-set persistent homology. 
In ordered-set homology we effectively restrict our chains to only contain ordered tuples were there are no repeats. We can still define homology, persistent homology and persistence modules. Furthermore, in some applications this may better reflect the connectivity structure  (such as in the analysis of the blue brain project in \cite{BBP}) but there are two important reasons why we are not considering order set persistence modules as a generalisation of the Rips persistence modules. The first reason is that when we restrict to symmetric functions we do not get isomorphic persistence modules to the standard Rips persistence modules. The second reason is that these persistence modules are not stable with respect to the correspondence distortion distance.

For example, consider the set $X=\{x,y\}$ with the $f$ the zero function. For $t<0$ then the corresponding ordered sets complexes are empty with  trivial homology. The ordered tuple complexes and Rips simplicial complexes are also empty and have trivial homology. For $t\geq0$, the corresponding ordered set complex consists of the ordered sets $(x)$, $(y)$, $(x,y)$ and $(y,x)$. It has non-trivial $1$-dimensional homology. To see this first observe that $(x,y)+(y,x)$ is a cycle but the space of $2$-chains is trivial so there are no non-trivial $1$-chain boundaries. 
In comparison, the Rips simplicial complex is $[x,y]$ which has no non-trivial $1$-cycles. The ordered-tuple complex is more complicated but everything ends up canceling each other. For example, this cycle of concern in the ordered set homology ($(x,y)+(y,x)$) is a boundary in the setting of OT-homology;  $(x,y)+(y,x)=\partial((x,y,x)+(x,x,x))$.

To see that the ordered set persistence modules are not stable with respect to the correspondence distortion distance compare $(X,f)$ in the example in the paragraph above to single point space $Z=\{z\}$ with function $g(z)=0$. The first dimensional ordered set homology for $Z$ is also trivial and so its first dimensional persistence module is also trivial. The correspondence $\{(x,z),(y,z)\}\subset X\times Z$ has zero distortion but the ordered set persistence modules are not $\epsilon$-interleaved for any $\epsilon$.

\section{Persistence modules via strongly connected components and preorder homology}
In this section we will consider constructions using an associated filtration of directed graphs or preorders. For each $t$ we can define a directed graph $\{D(X)_t\}$ where $x\to y$ is included in $D(X)_t$ when $\max\{f(x,y),f(x,x),f(y,y)\}\leq t$. From a directed graph we can induce a natural preorder via the existence of paths. That is a preorder where $x\leq y$ if there is a path from $x$ to $y$ in $D(X)_t$. We will construct persistence modules using the strongly connected components of the graphs $D(X)_t$, which are also the equivalence classes of the associated preorders. We also consider persistence modules using ordered-tuple complexes constructed over preorders.

Let us first introduce the construction of directed graphs and preorders from set-function pairs. 

\begin{definition}
Let $X$ be a set with a binary relationship $\leq$. Consider the following potential properties of $(X,\leq)$:
\begin{itemize}
\item[(i)] $x\leq x$ for all $x\in X$ (reflexive)
\item[(ii)] for all $x,y\in X$, if $x\leq y$ and $y\leq x$ then $x=y$ (antisymmetric)
\item[(iii)] for all $x,y,z\in X$, if $x\leq y$ and $y\leq z$ then $x\leq z$ (transitive)
\end{itemize}
We say that $(X,\leq)$ is a \emph{poset} is it satisfies $(i)$, $(ii)$ and $(iii)$. We say $(X,\leq)$ is a \emph{preorder} if it satisfies $(i)$ and $(iii)$.
\end{definition}

There is a natural equivalence relation on $X$ where $x\sim y$ when $x\leq y$ and $y\leq x$. If we quotient a preorder by this equivalence relation we are left with a poset.

One way to construct preorders is via directed graphs. Given a directed graph $G=(V,E)$ and vertices $x,y\in V$ we say there is a \emph{path} from $x$ to $y$ when there is a finite sequence of vertices $x_0, x_1, \ldots x_n=y$  such that $(x_i, x_{i+1})$ is a directed edge. To create a preorder on $V$ we declare that $x\leq y$ whenever there is a path from $x$ to $y$.  The strongly connected components of a directed graph are the equivalence classes of points where $v\sim w$ when there exists both a path from $v$ to $w$ and a path from $w$ to $v$.  Thus we see that the equivalence classes of this poset are precisely the strongly connected components of the directed graph it was built from. Suppose we start with a directed graph and we consider the preorder defined by the existence of paths. If we quotient by the equivalence relation to get a poset, then on the directed graph level we are collecting the vertices into the strongly connected components and then we have directed edges between these strongly connected components if there is a path between them. This will create an acyclic directed graph. 

We will first need to construct directed graphs from the sub-level sets of a set-function pair. From this we can consider filtrations of directed graphs and of preorders.

\begin{definition}\label{def:associated directed graphs}
Given a set-function pair $(X,f)$ there is a natural filtration of directed graphs $\{\D(X)_t:t\in [0,\infty)\}$ associated to $X$ by setting $\D(X,f)_t$ to the the directed graph with vertices $\{x\in X: f(x,x)\leq t\}$ and including the directed edge $x\to y$ whenever $\max\{f(x,x), f(y,y), f(x,y)\}\leq t$. We will call this the \emph{associated filtration of directed graphs} of $(X,f)$.
\end{definition}

It is necessary for the inclusion rule for the directed edges to occur at the maximum of $\{f(x,x), f(y,y), f(x,y)\}$ (rather than at $ f(x,y)$ which may occur earlier) to ensure that $D(X,f)_t$ will satisfy the closure conditions for a directed graph. In the case where $f=d$ a quasi metric then $d(x,x)=0=d(y,y)$ and $d(x,y)\geq 0$ and so the edge from $x$ to $y$ is included at $t=d(x,y)$.

We define a filtration of preorders is parameterised family of preorders $\{(X_t, \leq_t): t\in \R\}$ such that for all $s\leq t$ we have $X_s\subset X_t$ and if $x,y\in X_s$ with $x\leq_s y$ then $x\leq_t y$. From a filtration of associated graphs for a set-function pair we can construct a natural filtration of preordered spaces as follows. 

\begin{definition}\label{def:posetOT}
Let $(X,f)$ be a set-function pair and let $\{D(X,f)_t\}$ be its associated filtration of directed graphs. For each $t\geq 0$  construct a preordered space $(X_t, \leq_t)$ with $X_t$ the set of points in $D(X,f)_t$ and $x\leq _t y$ when there exists a path in $D(X,f)_t$ from $x$ to $y$. We call this the \emph{associated filtration of preorders}.
\end{definition}

The following is a useful lemma for proving the interleaving results for the persistence modules constructed with strongly connected components or with preorder homology.

\begin{lemma}\label{lem:path}
	Let $X$ and $Y$ be sets and $(X,f)$ and $(Y,g)$ be set-function pairs with $d_{CD}((X,f),(Y,g)) $ finite. Let $\D(X,f)=\{D(X,f)_t\}$ and $\D(Y,g)=\{D(Y,g)_t\}$ be the associated filtrations of directed graphs. Let $\M\subset X\times Y$ be a correspondence with $dis_{(X,f),(Y,g)}(\M)$ finite
\begin{enumerate}[(i)]
\item If $(x,y)\in \M$ and  $x\in D(X,f)_t$ then $y\in D(Y,g)_{t+dis(\M)}.$
\item If $(x_1, x_2), (y_1, y_2)\in \M$ and there exists a directed path from $x_1$ to $x_2$ in $D(X,f)_t$ then there exists a directed path from $y_1$ to $y_2$ in $D(Y,g)_{t+dis(\M)}$.
\end{enumerate}
\end{lemma}

\begin{proof}
\begin{enumerate}[(i)]
\item If $x\in D(X,f)_t$ then $f(x,x)\leq t$. Since $(x,y)\in \M$ we know $g(y,y) \leq t+dis(\M)$ and hence $y\in D(Y,g)_{t+dis(\M)}.$
\item Suppose that there is a path from $x_1$ to $x_2$ in $D(X,f)_t$. This means that there exists a sequence of points $(x_1=a_1, a_2, \ldots, a_k=x_2)$ in $X$ such that $f(a_i, a_{i+1})\leq t$. There exists a sequence of points in $Y$, $y_1=b_1,b_2, \ldots, b_k=y_2$ where $(a_i,b_i)\in \M$.  By (i). we know each of the $b_i$ lie in $D(Y,g)_{t+dis_{(X,f),(Y,g)}(\M)}$. Since each $(a_i, b_i)\in \M$ we have 
$$|f(a_i, a_{i+1}) - g(b_i, b_{i+1})|\leq dis_{(X,f),(Y,g)}(\M)$$ for each $i$ and hence $(y_1=b_1, b_2,\ldots b_k=y_2)$ is a path in $D(Y,g)_{t+dis_{(X,f),(Y,g)}(\M)}$.
\end{enumerate}
\end{proof}

The lemma can be rewritten in terms of preorders; for $(x_1, y_1), (x_2, y_2)\in \M$ if $x_1\leq_t^f x_2$ then $y_1\leq_{t+dis_{(X,f),(Y,g)}(\M)}^g y_2$.

\subsection{Strongly connected components persistence}

%
%

Dimension $0$ persistent homology is all about tracking the evolution of connected components. For directed graphs, unlike graphs, there is choice in how to interpret what a connected component is, and each interpretation providing their own corresponding persistence module. Here we will consider the persistence of weakly and strongly connected components. 

Weakly connected components are the components of the graph when the directions are forgotten. Given a filtration of a directed graph by edge weights, the weakly connected persistence would be the same as the dimension $0$ persistent homology of the Rips filtration under $\sym_1$ in section \ref{sec:sym}, and to the dimension $0$ directed Rips persistence module in section \ref{sec:OT}. 

Studying strongly connected components  will provide new information. Recall the strongly connected components of a directed graph are the equivalence classes of points where $v\sim w$ when there exists both a path from $v$ to $w$ and a path from $w$ to $v$. 
Given a filtration of directed graphs we can construct a persistence module based on linear combinations of strongly connected components (analogous to dimension $0$ homology being interpreted as the space of formal linear combinations of connected components).

\begin{definition}
We call $\D=\{D_t: t\in\R \}$ a filtration of directed graphs if $D_t$ is directed graph for all $t$ such that if $s\leq t$ then $D_s$ is a directed subgraph of $D_t$.  Given a filtration of directed graphs $\D=\{D_t\}$, let $[v]_t$ denote the strongly connected component of $D_t$ containing $v$. We define the \emph{strongly connected persistence module} corresponding to $\D$  as follows:
\begin{itemize}
\item for each $t\in \R$ set the vector space $V_t$ to be the vector space of finite linear combinations of strongly connected component (elements are of the form $\sum_{i=1}^k \lambda_i [v_i]_t$ with $\lambda_i \in \F$)
\item for each pair $t\leq s$ we have a linear map induced from inclusion $\iota_{t\to s}( \sum_{i=1}^k \lambda_i [v_i]_t)= \sum_{i=1}^k \lambda_i [v_i]_s$
\end{itemize}
\end{definition}

We will now check that the strongly connected component persistence module does satisfy the requirements of a persistence module. Whenever we have an inclusion of directed graphs $D_t \subset D_s$, whenever there is a path from  $v$ to $w$ in $D_t$ then there is also a path from  $v$ to $w$ in $D_s$. This implies that the maps $\iota_{t\to s}$ are well defined.
Furthermore for $u\leq t\leq s$ we have $\iota_{t\to s}(\iota_{u \to t}( \sum_{i=1}^k \lambda_i [v_i]_u))= \sum_{i=1}^k \lambda_i [v_i]_s=\iota_{u\to s}( \sum_{i=1}^k \lambda_i [v_i]_u)$. 
Whenever the directed graphs $D_t$ are all finite (which is true in almost any application) we automatically know that the $V_t$ are all finite dimensional and hence the strongly connected persistence module is tame.

We can create strongly connected persistence modules from set-function pairs via its associated filtration of directed graphs.

\begin{theorem}\label{theorem:scc}
Let $X$ and $Y$ be sets and $(X,f)$ and $(Y,g)$ be set-function pairs with $d_{CD}((X,f),(Y,g)) $ finite. Let $\D(X,f)=\{D(X,f)_t\}$ and $\D(Y,g)=\{D(Y,g)_t\}$ be the associated filtrations of directed graphs. Let $\P^X$ and $\P^Y$ be the strongly connected component persistence modules for $\D(X,f)$ and $\D(Y,g)$. Then $d_{int}(\P^X,\P^Y)\leq d_{CD}((X,f),(Y,g)) $.
\end{theorem}

\begin{proof}
Fix a correspondence $\M\subset X\times Y$ with $dis_{(X,f),(Y,g)}(\M)$ finite.

Construct a map $\alpha: X \to Y$ where for each $x$ we arbitrarily fix a representative from $\{y\in Y: (x,y)\in \M\}$ and construct a map $\beta: Y \to X$ where for each $y$ we arbitrarily fix a representative from $\{x\in X: (x,y)\in \M\}$. 

If $[x_1]_t=[x_2]_t$ then there exist paths in $D(X,f)_t$ from $x_1$ to $x_2$ and from $x_2$ to $x_1$. By Lemma \ref{lem:path} there exist paths in $D(Y,g)_{t+dis_{(X,f),(Y,g)}(\M)}$ from $\alpha(x_1)$ to $\alpha(x_2)$ and from $\alpha(x_2)$ to $\alpha(x_1)$. This means that $\alpha$ induces a well-defined linear map 
\begin{align*}
\alpha_*:P^X_t& \to P^Y_{t+dis_{(X,f),(Y,g)}(\M)}\\
[x]_t &\mapsto [\alpha(x)]_{t+2dis_{(X,f),(Y,g)}(\M)}.
\end{align*}
Similarly $\beta$ induces a linear map $\beta_*:P^Y_t \to P^X_{t+2dis_{(X,f),(Y,g)}(\M)}$ where $[y]_t \mapsto [\beta(y)]_{t+2dis_{(X,f),(Y,g)}(\M)}$. 

It only remains to show that $\alpha_*$ and $\beta_*$ satisfy an $2dis_{(X,f),(Y,g)}(\M)$ interleaving. That \eqref{eq:comdia1} commutes follows directly from the construction of $\alpha$ and $\beta$.


Let $f(x,x)=t$ and hence $x\in D(X,f)_t$. From our construction of $\alpha$ and $\beta$ we know that $(x, \alpha(x))$ and $(\beta(\alpha(x)), \alpha(x))$ are both in $\M$. By Lemma \ref{lem:path} this implies that there are directed paths in both directions between $\beta(\alpha(x))$ and $x$ in the directed graph  $D(X,f)_{t+2dis_{(X,f),(Y,g)}(\M)}$; and hence they lie in the same strongly connected component.


Similarly, for every $y\in Y$ with $g(y,y)=t$, we know that $\alpha(\beta(y))$ and $y$ lie in the same strongly connected component in $D(Y,g)_{t+2dis_{(X,f),(Y,g)}(\M)}$. This ensures that we satisfy  \eqref{eq:comdia2}.

By taking the infimum over all correspondences we see that the interleaving distance between $P^X$ and $P^Y$ is bounded above by $2d_{CD}((X,f),(Y,g))$.
\end{proof}

We provide some pseudocode (Algorithm 1 in the appendix) for an algorithm that computes the interval decomposition of the strongly connected component persistence module from a filtration of directed graphs. It is a modification of the union-find algorithm used to compute the standard dimension $0$ persistent homology.  In the union-find algorithm each connected component is represented by a root vertex with an additional data of its birth time. The main difference for strongly connected components is that we have to also keep track of when directed paths exist between the various strongly connected components. These are stored as a list the root vertex of ``in'' and ``out'' connected components. Here ``in'' means a connected component that there is a path pointing into the current component and ``out'' means there is a path pointing out of the current component . Note that for any root vertex these in and out sets are disjoint, as being in both would imply they are the same strongly connected component. The main challenge in this modification is to ensure that at each stage the list of ``in'' and ``out'' strongly connected components listed by the root vertices are referred to by their root vertex.

%
%
%

%
%
\subsection{OT-complexes constructed using the preorder structure}\label{sec:poset}

In the theory of partially ordered sets (``posets''), the order complex of a poset is the set of all finite chains. Its homology contains important information about the poset. Preorders are a generalisation of posets where we drop the antisymmetry condition. Poset homology naturally extends to preorders, where we will call it preorder homology. It is easier and more flexible to construct filtrations of preorders than of posets.

From the associated filtration of directed graphs of a set-function pair we can create a filtration of preorders which we will call the preorder Rips filtration. From the filtration of preorders we can construct persistence modules using preorder homology to generate preorder Rips persistence modules. These persistence modules enjoys stability with respect to the correspondence distortion distance. The homology dimension $0$ preorder Rips persistence module is isomorphic to that of its weakly connected components, its directed Rips persistence module and the standard Rips persistence module under $\sym_1$. If the input is a symmetric function then its higher dimensional preorder Rips persistence modules are all trivial, showing that preorder Rips persistence module describes asymmetry information.

In this paper we will generalise to preorders some constructions normally defined for posets. The homology of a poset has been defined and studied via its corresponding Alexandrov topology. Preorders are in bijective correspondence with Alexandrov topologies, with the antisymmetry condition (which is the axiom that makes a preorder a poset) translating to those topologies that are $T_0$. For each preorder there is a canonical poset over its equivalence classes, and this poset corresponds to the Kolmogorov quotient of the Alexandrov topology of that original preorder. Because these quotient spaces are weakly homotopy equivalent, standard references for Alexandrov topology often state they will restrict their analysis to  $T_0$ spaces/ posets (e.g. \cite{McCord, May}) . It is for this reason that definitions are usually only stated for posets and not more generally for preorders. In the appendix we will go into more detail into this background material and justify why the definitions given in this section are the natural generalisation of those traditionally given for posets.


Let us now construct a OT-complex from a preorder. 
\begin{definition}
Given a preorder $(X, \leq)$, let $\O(X, \leq)$ be the OT-complex containing $(x_0, x_1, \ldots x_p)$  when $x_0 \leq x_1 \leq \ldots \leq x_p$. 
We call $\O(X,\leq)$ the \emph{preorder OT-complex} of $(X,\leq)$. 
\end{definition}


\begin{definition}
Given a preorder $(X,\leq)$, its associated order complex $\Delta(X,\leq)$ is an abstract simplicial complex whose vertices are the elements of $X$ and whose faces are the chains (subsets where each pair is comparable) of $(X,\leq)$. 
\end{definition}

From a filtration of preorders we can construct a filtration of OT-complexes. From this persistence modules can be constructed as standard with OT-homology classes as the vector spaces and induced maps from inclusions as the transition maps. 

\begin{definition}
Let $\O(X,f)=\{\O(X,f)_t\}$ be the filtration of OT-complexes corresponding to  the filtration of posets $\{(X_t, \leq_t)\}$. We call $\O(X,f)$ the \emph{preorder filtration} of $(X,f)$. 
\end{definition}


 
 In the appendix we see that the simplicial homology of the order complex $\Delta(X,\leq)$ is naturally isomorphic to the homology of preorder OT-complex $\O(X,\leq)$. Moreover, isomorphisms between the simplicial homology of the order complexes and the homology of the preorder OT-complexes will extend to persistent homology as they commute with the maps on homology induced by inclusions.
 
 \begin{definition}
We define the $k^{th}$ dimensional \emph{preorder persistence module} corresponding to the filtration of preorders $\mathcal{X}=\{(X_t, \leq_t)\}$ as the dimension $k$ OT-homology persistence module for the filtration of OT-complexes $\{\O(X,\leq_t)\}_{t\in \R}$.
 \end{definition}
 

 
Just as in the previous constructions in this paper we can prove that the corresponding persistence modules built from functions $f:X\times X\to \R$ and $g:Y\times Y\to \R$ are stable with respect to the correspondence distortion distance.

\begin{theorem}\label{theorem:poset}
Let $(X,f)$ and $(Y,g)$ be set-function pairs with preorder Rips filtrations  $\O(X,f)$ and $\O(Y,g)$.
Let $\P^X$ and $\P^Y$ be the $k$th dimensional persistence modules for $\O(X,f)$ and $\O(Y,g)$ respectively.
 Then $d_{int}(\P^X, \P^Y)\leq 2d_{CD}((X,f),(Y,g))$.
\end{theorem}
\begin{proof}
%
%
%

Since $d_{CD}((X,f),(Y,g)) $ is finite there exists some correspondence $\M\subset X\times Y$ with $dis_{(X,f), (Y,g)}(\M)$ finite. Fix a correspondence $\M\subset X\times Y$ with $dis_{(X,f), (Y,g)}(\M)$ finite. From this correspondence construct preorder Rips filtrations $\{\O(\M,f^\M)_t\}$ and $\{\O(\M,g^\M)_t\}$ with corresponding  $k$-th dimensional persistence modules $P^{(X,\M)}$ and $P^{(Y,\M)}$.

By construction $\{\O(\M,f^{\M})_t\}$ is an expansion of $\{\O(X,f)_t\}$ and thus by Proposition \ref{prop:expansion} we know that the persistence modules $P^X$ and $P^{(X,\M)}$ are isomorphic. Similarly we can also show that  $P^Y$ and $P^{(Y,\M)}$ are isomorphic.

If $((x_0,y_0), (x_1,y_1), \ldots, (x_n,y_n))\in \O(\M,f^{\M})_t$ then $x_0, \ldots , x_n\in D(X)_t$ and there exist directed paths from $x_i$ to $x_j$ in $D(X,f)_t$ for all $i\leq j$. By Lemma \ref{lem:path} there must exist a directed paths from $y_i$ to $y_j$ in $D(Y,g)_{t+dis_{(X,f), (Y,g)}(\M))}$ for all $i\leq j$.  This implies that $ \O(\M,f^{\M})_t \subset \O(\M,g^{\M})_{t+dis_{(X,f), (Y,g)}(\M))}$ for all $t$. Similarly $ \O(\M,g^{\M})_t \subset \O(\M,f^{\M})_{t+dis_{(X,f), (Y,g)}(\M))}$. 

These inclusion maps induce a $dis_{(X,f), (Y,g)}(\M))$ interleaving between $P^{(X,\M)}$ and $P^{(Y,\M)}$. This implies that $P^{X}$ and $P^{Y}$ are  $dis_{(X,f), (Y,g)}(\M))$ interleaved.

By considering the infimum of the interleavings constructed by correspondences we see that the interleaving distance between $P^X$ and $P^Y$ is bounded above by $2d_{CD}((X,f),(Y,g))$.
\end{proof}

As shown in Theorem \ref{thm:preorderhomology} (in the Appendix), the simplicial homology of the order complex is naturally isomorphic to the OT-homology of $\O(X,\leq)$. Furthermore, this isomorphism result holds also for persistence modules of filtrations of simplicial complexes as the isomorphisms on homology groups commute with the induced maps on homology by inclusions. This implies that interval decomposition of the $k^{th}$ preorder persistence modules can be computed via the simplicial persistent homology over the filtration of simplicial complexes $\{\Delta(X_t, \leq^f_t)\}$. 

\section{Future Directions}
There are many future directions related to the research in this paper. Examples include: 
\begin{itemize}
\item Applying the constructions in this paper to quasi-metric spaces to see what they reveal about their quasi-metric structure, or to use as a method of getting a lower bound on the correspondence distortion distance distance between different quasi-metric spaces.
\item To adapt these methods to construct persistence modules for sub-level set filtrations of special functions on quasi-metric spaces and proving related stability results. For example, we conjecture that the all four constructions
built from a suitably defined sublevel sets of the extremity function of a quasi-metric space (analogous to constructions in \cite{chazal2009gromov}) could have correspondence distortion distance stability with respect to the original quasi-metric distance functions. This would provide another way of capturing the ``shape'' of a quasi-metric space. 
\item Finding nice sufficient conditions on functions $f:X \times X \to \R$, with $|X|$ infinite, as to when these various Rips constructions create tame persistence modules. Even when restricting to the case of quasi-metric spaces it is not even clear how we should define an $\epsilon$-sampling or compactness. In the symmetric case, definitions have been used to describe sufficient conditions for metric spaces that result in tame persistence modules (such as in \cite{CSO}). 
\item Algorithmic techniques for computing OT persistent homology efficiently. In particular is there a related filtration of simplicial complexes that have the isomorphic OT- persistent homology, at least in low homology dimensions?
\end{itemize}
\section{APPENDIX}

\subsection{Algorithm to compute interval decomposition of the  strongly connected persistence module
}

\begin{algorithmic}[1]
\INPUT List $L$ of vertices $V=\{v_1, v_2, \ldots v_n\}$ and directed edges $\{(v_{i_1}\to v_{j_1}), \ldots (v_{i_m}\to v_{j_m})\}$, each with a real valued height such that $h(v_i\to v_j)\geq \max\{h(v_{i}), h(v_j)\}$. These vertices and directed edges are ordered in a combined list $L$ by increasing height values. All the vertices at a height value occur before the edges at that same height.
\OUTPUT Interval decomposition of the strongly connected component persistence module from filtration of sublevel sets of the height function
\Function{Find}{$x$}
\While{root($x$)$!=x$} 
		\State{$x=$root$(x)$}
\EndWhile
 \State \textbf{return} $x$
\EndFunction
\\
\Procedure{Union}{$v_{tail}$, $v_{head}$, height}
\State $W=v_{tail}.in \cap v_{head}.out$
\State $\hat{w}:=$ earliest $w\in W$ to appear in list $L$
\For{$w\in W$, $w\neq \hat{w}$}
		\State $root(w)=\hat{w}$\;
		\If{$h(w)<height$}	
			\State append $[h(w), height)$ to BARCODE
		\EndIf
		\State $\hat{w}.in=\{$\Call{Find}{$x$} for $x\in v_{tail}.in\}$ \Comment A s.c.c. has a path to $\hat{w}$ iff it has a path to $v_{tail}$
		\State $\hat{w}.out=\{$\Call{Find}{$x$} for $x\in v_{head}.out\}$\Comment A s.c.c. has a path from $\hat{w}$ iff it has a path from $v_{head}$
		\For{$x\in \hat{w}.in$}
			\State$x.out=\{$\Call{Find}{y} for $ y \in x.out\cup \hat{w}.out\}$ 
		\EndFor
		\For{$x\in \hat{w}.out$}
			\State $x.in=\{$\Call{Find}{y} for $ y \in x.in\cup \hat{w}.in \}$ 
		\EndFor
		\EndFor
\EndProcedure
\\
\Procedure{UpdateInOut}{$v_{tail}$, $v_{head}$, height}
	\For{$x\in v_{tail}.in$}
	  	\State$x.out=\{$\Call{Find}{y} for $ y \in  v_{head}.out \cup x.out\}$
	\EndFor
	\For{$y\in v_{head}.out$}
		\State$y.in=\{$\Call{Find}{x} for $ x \in  v_{tail}.in\cup y.in \}$	
	\EndFor
\EndProcedure
\\
\For{$i=1$ to $length(L)$}
	\If{$L(i)$ is a vertex $v_k$}
		\State Add a vertex to $A$. Label it with (height$=h(L(i))$, root$=v_k$, in$=\{v_k\}$, out$=\{v_k\}$)
		\If{$L(i)$ is a directed edge $v_j \to v_k$}
		\State $v_{tail}:=\Call{Find}{v_j}$, $v_{head}:=\Call{Find}{v_k}$
			\If{$v_{head} \notin v_{tail}.out$}
				\If{$v_{head} \notin v_{tail}.in$} \Comment We need to update the paths between s.c.c.s 
					\State \Call{UpdateInOut}{$v_{tail}$, $v_{head}$}
				\EndIf
				\If{$v_{head} \in v_{tail}.in$} 	\Comment This is when various s.c.c.s merge.				
					\State \Call{Union}{$v_{tail}$, $v_{head}$, $h(L(i))$}
				\EndIf
			\EndIf
		\EndIf
	\EndIf
\EndFor			

	\State RemainingComponents:=$\{\Call{Find}{x} \text{ for } x \in V\}$\Comment Final set of strongly connected components
	\For{$x\in$ RemainingComponents}
	\State Append $[h(x),\infty)$ to BARCODE
	\EndFor
\end{algorithmic}

\subsection{Homology of a poset}
There are multiple ways to compute the homology of a poset, including via Alexandrov topological spaces and order simplicial complexes.  For each preorder there is a canonical poset we call its equivalence class poset.  In this appendix we show that the definitions of homology of a poset can naturally be extended to preorders. Furthermore, the resulting homology of a preorder is naturally isomorphic to the homology of its equivalence class poset. This justifies the constructions in section \ref{sec:poset}.

An \emph{Alexandrov topology} is a topology in which the intersection of any family of open sets is open. It is an axiom of topology that the intersection of any finite family of open sets is open; in Alexandrov topologies the finite restriction is dropped.  Given an Alexandrov topology we can construct a special preorder called its specialisation preorder. 

\begin{definition}
Let $X =(X,\tau)$ be an Alexandrov space.
The \emph{specialisation preorder} on X is the preorder where $x\leq y$ if and only if $x$ is in the closure of $\{y\}$.
\end{definition}

In the other direction, given an preorder $(X,\leq)$ there is a unique Alexandrov topology whose specialisation preorder is $(X,\leq)$. To construct this, let the open sets $\tau$ on $X$ be the upper sets:
$$\tau=\{ U\subset X: \forall x,y \in X, \text{ if } x\leq y \text{ and } x\in U \text{ then } y\in U\}.$$
The corresponding closed sets for $\tau$ are the lower sets:
$$ \{\,S\subseteq X:\forall x,y \in X, \text{ if }  x\in S \text{ and } y\leq x \text{ then } y\in S\,\}.$$ The topology $\tau$ is generated by the sets $U_x=\{y:x\leq y\}$. 

A topological space $X$ is a \emph{$T_0$ space} if for any pair of points in $X$ there exists an open set containing one and only one of them. It is an exercise to see how the antisymmetry condition of posets directly corresponds to the Alexandrov topologies that are $T_0$.  

We can construct $T_0$ spaces by taking Kolmogorov quotients. The \emph{Kolmogorov quotient} of a topological space is defined as its quotient by the equivalence relation of topological indistinguishability, equipped with the quotient topology.  

There is a natural way of constructing a poset from a preorder by using quotients. For $(X,\leq)$ a preorder define an equivalence relation $x\sim y$ when $x\leq y$ and $y\leq x$. Let $\tilde{X}=X/\sim$ be the quotient space on these equivalence classes. It is easy to check that the binary relation $\leq$ is now well defined on $\tilde{X}$ and that $(\tilde{X},\leq)$ is a poset. We will call $(\tilde{X},\leq)$ the \emph{equivalence class poset} of $(X,\leq)$. The following lemma states the relationship between a preorder and its equivalence class poset is analogous to taking the Kolmogorov quotient of its Alexandrov topology. The proof for finite spaces is Lemmas 8 and 9 in \cite{McCord}, and the extension to infinite spaces can be proved similarly (see \cite{kukiela2010homotopy}).

\begin{lemma}\label{lem:Kolmogorov}
Let $(X,\leq)$ be a finite preorder with equivalence class poset $(\tilde{X},\leq)$. The Alexandrov topology of $(\tilde{X},\leq)$ is the Kolmogorov quotient of the Alexandrov topology of $(X,\leq)$. Furthermore, the Alexandrov topologies of  $(X,\leq)$ and $(\tilde{X},\leq)$ are homotopy equivalent.
\end{lemma}

%
Since homology is defined up to weak homotopy equivalence, often in analysis researchers restrict their analysis from general topological spaces to $T_0$-spaces as they do not lose any homological information by taking the Kolmogorov quotient. Thus many definitions are stated as for posets even though they could be defined for all preorders.

One definition of the homology of a poset is the singular homology of the Alexandrov topologies which has that poset as its specialisation order. Since the specialisation orders of Alexandrov topologies provide a one to one correspondence between Alexandrov topologies and preorders, we can generalise this to define the homology of a preorder as the singular homology of the Alexandrov topologies which has that preorder as its specialisation order.

A \emph{chain} in a poset is defined as a subset of elements which are all pairwise comparable. Note that there is no order of the elements given as part of the information of the chain but that the transitivity of a preorder will ensure that there exists a total ordering of any chain. In a poset the antisymmetry condition ensures that this order is unique. In a general preorder multiple possible orders might be possible.

In a poset, the unique ordering of elements in a chain means we can define chain complexes and homology groups for a poset directly via chains.
We thus say that an \emph{$m$-chain} of a poset $P$ is a totally ordered subset $c = (x_0 < x_1 < \ldots < x_{m})$ of $P$ written in order. We can construct a chain complex by setting $C_j(P, R)$ to be the $R$-module freely generated by $j$-chains, and defining boundary maps $\partial_j: C_j(P)\to C_{j-1}(P)$ by  $\partial_j(x_0 < x_1 < \ldots < x_{m})=\sum_{i=0}^{m} (-1)^{i} (x_0<x_1<\ldots \hat{x_i} \ldots <x_m)$ and extending linearly.

We can observe that this chain complex is exactly that for ordered sets (see subsection \ref{sec:OS}). If we specify the order of each chain, we can extend this definition to preorders as the OS-homology. Generally the OS-homology and the OT-homology are not isomorphic (see section \ref{sec:OS}). However, in the special case of posets they do define isomorphic homology groups, as proved below in Theorem \ref{prop:equiv}.

An alternative definition for the homology of a poset is via the construction of its associated order simplex. The associated order complex $\Delta(X,\leq)$ for the poset $(X,\leq)$ is the abstract simplicial complex whose vertices are the elements of $X$ and whose faces are the chains (subsets where each pair is comparable) of $(X,\leq)$. The definition of the associated order complex of a preorder given in section \ref{sec:poset} restricts to the standard definition for posets.  

The following theorem presents some relationships between these different homology constructions.

\begin{theorem}\label{prop:equiv}
Let $(\tilde{X},\leq )$ be a poset. The following homology groups are all isomorphic: 
\begin{itemize}	
\item[(i)] OS-homology of the finite chains of $(\tilde{X}, \leq)$
\item[(ii)] OT-homology of preorder OT-complex $\O(\tilde{X},\leq)$
\item[(iii)] simplicial homology of the order complex $\Delta(\tilde{X},\leq)$ 
\item[(iv)] singular homology of the Alexandrov topology with specialisation order $(\tilde{X}, \leq)$
\end{itemize}
\end{theorem}
\begin{proof}

The proof that $(ii)$ and $(iv)$ are isomorphic is in Theorem 2 in \cite{McCord}. The isomorphism between $(i)$ and $(iii)$ is via the unique total orderings of each simplex in the order complex. It is the induced map on homology of $(x_0< x_1<\ldots x_k) \mapsto [x_0, x_1, \ldots x_k]$. 

We will now prove that $(i)$ is isomorphic to $(ii)$.
The set of ordered tuples forms a basis $B$ for $OT(\tilde{X},\leq)$. Define $\Phi:B \to \{\text{subcomplexes of } OT(\tilde{X},\leq)\}$ by setting $\Phi(\tau)$ to be the subcomplex of $OT(\tilde{X},\leq)$ containing only ordered tuples with elements within $\tau$. Since $\tau$ is an ordered tuple, it has a smallest element $x$ and for any $\alpha\in \Phi(\tau)$ the ordered tuple concatenating $x$ in front of $\alpha$ (which we will denote by $(x\alpha)$) is also an element in $\Phi(\tau)$. Given a boundary $\alpha$, we can see that $\partial(x\alpha)=\alpha -(x\partial(\alpha))=\alpha$. This implies that $\Phi$ is an acyclic carrier. 

Set $f:OT(\tilde{X},\leq)\to OT(\tilde{X},\leq)$ by $f(\tau)$ the identity when $\tau$ does not contain repeats (i.e. lives in $OS(\tilde{X},\leq)$ and $f(\tau)=0$ if $\tau$ contains a repeat). $f$ commutes with the boundary map because all repeats of a particular element in a tuple must be consecutive when working with posets. It is this claim that does not hold more generally between OT-complexes and OS-complexes. Since both $f$ and the identity map are both carried by $\Phi$ the acyclic carrier theorem (see \cite{munkres1984elements}) ensures that $f$ and the identity map are chain homotopic and hence the OS-homology of the finite chains of $(\tilde{X}, \leq)$ and the  OT-homology of preorder OT-complex $\O(\tilde{X},\leq)$ are isomorphic.

\end{proof}

Each of these four different constructions of homology groups for posets can be generalised to preorders. Three of these generalise in a way that the homology groups are invariant under taking equivalence class posets (or equivalently under taking Kolmogorov quotients). The OS-homology of chains is the odd one out in this respect. A counter example is the preorder $X=\{x,y\}$ with $x\leq y$ and $y\leq x$. It has non-trivial OS-homology in dimension one but its equivalence class poset $\tilde{X}=\{[x]\}$ has trivial OS-homology in dimension one.

\begin{theorem}\label{thm:preorderhomology}
Let $(X,\leq)$  be a preorder with equivalence class poset $(\tilde{X},\leq)$.Then
\begin{itemize}
\item[(a)] The preorder OT-complex $\O(X,\leq)$ is an expansion of $\O(\tilde{X},\leq)$ and hence has the same OT-homology. 
\item[(b)] There is a natural projection map from $\Delta(X,\leq)$ to $\Delta(\tilde{X}, \leq)$. This projection map induces an isomorphism on their simplicial homology groups.
\item[(c)] The singular homology of the Alexandrov topology with specialisation order $(X, \leq )$ is isomorphic to the singular homology of the Alexandrov topology with specialisation order $(\tilde{X}, \leq )$. 
\end{itemize}
\end{theorem}

\begin{proof}
(a) The OT-complexes $\O(X,\leq)$ and $\O(\tilde{X},\leq)$ are closed under adjacent repeats by construction. The quotient maps sending $X$ to its equivalence class poset $\tilde{X}$ shows that  $\O(X,\leq)$ is an expansion of $\O(\tilde{X},\leq)$. We conclude that they are isomorphic by applying Proposition \ref{prop:expansion}.

(b) Construct a map $f:\tilde{X} \to X$ by fixing a representative $x\in X$ for each equivalence class $[x]\in \tilde{X}$. We can embed $\Delta(\tilde{X}, \leq)$ into $\Delta(X, \leq)$ via the induced map of $f$. A straight line homotopy provides a deformation from $\Delta(X, \leq)$ to $f(\Delta(\tilde{X}, \leq))$. The result then follows because deformation retractions induce isomorphisms on homology classes.

(c) The proof follows from Lemma \ref{lem:Kolmogorov} as homotopic topological spaces have isomorphic singular homology groups.
\end{proof}

Combining these theorems we conclude that the OT-homology of preorder OT-complexes, simplicial homology of the associated order complex of a preorder and the singular homology of the Alexandrov topology of a preorder are all isomorphic. These isomorphisms extend to persistent homology as they commute with the maps on homology induced by inclusions.

%
%
%
%

\end{document}